\documentclass[a4paper,12pt]{amsart}

\usepackage{latexsym}
\usepackage{amssymb}
\usepackage{graphicx}
\usepackage{amsthm}
\usepackage{epsfig}

\usepackage{amscd,enumerate}
\usepackage{amsfonts}

\usepackage[all]{xy}

\input xy
\xyoption{all}

\newtheorem{lemma}{Lemma}[section]
\newtheorem{corollary}[lemma]{Corollary}
\newtheorem{theorem}[lemma]{Theorem}
\newtheorem{proposition}[lemma]{Proposition}

\newtheorem{examples}[lemma]{Examples}

\newtheorem{conjecture}[lemma]{Conjecture}

\begin{document}


\title[$\ast$-Regular Leavitt path algebras of arbitrary graphs]{$\ast$-Regular Leavitt path algebras
 of arbitrary graphs}

\author{G. Aranda Pino}
\address{Aranda Pino: Departamento de \'Algebra, Geometr\'\i a y Topolog\'\i a, Universidad de M\'alaga,
 29071 M{\'a}laga, Spain}
\email{g.aranda@uma.es}
\author{K. M. Rangaswamy}
\address{Rangaswamy: Department of Mathematics, University of Colorado, Colorado Springs,
 CO 80933, USA}\email{ranga@uccs.edu}
\author{L. Va\v{s}}
\address{Va\v{s}: Department of Mathematics, Physics and Statistics, University of the Sciences in
Philadelphia, Philadelphia, PA 19104, USA}
\email{l.vas@usp.edu}

\begin{abstract}
If $K$ is a field with involution and $E$ an arbitrary graph, the involution from $K$ naturally induces an
involution of the Leavitt path algebra $L_K(E).$ We show that the involution on $L_K(E)$ is proper if the
involution on $K$ is positive definite, even in the case when the graph $E$ is not necessarily finite or
row-finite.

It has been shown that the Leavitt path algebra $L_K(E)$ is regular if and only if $E$ is acyclic. We give
necessary and sufficient conditions for $L_{K}(E)$ to be $^\ast$-regular (i.e. regular with proper
involution). This characterization of $^\ast$-regularity of a Leavitt path algebra is given in terms of an
algebraic property of $K,$ not just a graph-theoretic property of $E.$ This differs from the known
characterizations of various other algebraic properties of a Leavitt path algebra in terms of graph-theoretic
properties of $E$ alone.

As a corollary, we show that Handelman's conjecture (stating that every $^\ast$-regular ring is unit-regular)
holds for Leavitt path algebras. Moreover, its generalized version for rings with local units also continues
to hold for Leavitt path algebras over arbitrary graphs.
\end{abstract}

\maketitle

\section*{Introduction}

Leavitt path algebras can be regarded as the algebraic counterparts of the graph $C^*$-algebras, the
descendants from the algebras investigated by J. Cuntz in \cite{Cuntz}. Leavitt path algebras can also be
viewed as a broad generalization of the algebras constructed by W. G. Leavitt in \cite{Leavitt} to produce
rings without the Invariant Basis Number property.

The Leavitt path algebra $L_K(E)$ was introduced in the papers \cite{AA1} and \cite{AMP}. $L_K(E)$ was first
defined for a row-finite graph $E$ (countable graph such that every vertex emits only a finite number of
edges) and a field $K$. Although their history is very recent, a flurry of activity has followed the papers
\cite{AA1} and \cite{AMP}. The main directions of research include: characterization of algebraic properties
of a Leavitt path algebra $L_K(E)$ in terms of graph-theoretic properties of $E$; study of the modules over
$L_K(E)$; computation of various substructures (such as the Jacobson radical, the center, the socle and the
singular ideal); investigation of the relationship and connections with $C^*(E)$ and general $C^*$-algebras;
classification programs; study of the $K$-theory; and generalization of the constructions and results first
from row-finite to countable graphs and finally, from countable to completely arbitrary graphs.  For examples
of each of these directions see \cite{AA1, APS, AMMS1, AT, AbAnhLouP, ABC, Tomforde, G}.

The base field $K$ is naturally endowed with an involution $^-$ (the identity involution can always be
considered in the absence of other possibilities). A given involution on $K$ naturally induces an involution
of a Leavitt path algebra $L_K(E).$ The presence of an involution on a ring yields some favorable features:
the ring is isomorphic to its opposite ring and a certain dose of symmetry is present. For example, a left
Rickart $^\ast$-ring is also a right Rickart $^\ast$-ring while this is not the case for one-sided Rickart
rings. Also, consideration of a complex Leavitt path algebra $L_{\mathbb{C}}(E)$ as a an algebra with
involution (induced from the complex-conjugate involution on $\mathbb{C}$), brings $L_{\mathbb{C}}(E)$ a step
closer to its analytic counterpart $C^*(E).$ These facts justify our interest in the study of the involution
on a Leavitt path algebra.

Iain Raeburn in \cite{Iain} shows that the involution on $L_K(E)$ is proper if the involution on $K$ is
positive definite and $E$ is a row-finite countable graph without sinks. We extend this result to arbitrary
graphs (Proposition \ref{involutionproperarbitrarygraphs}). We also show that the converse holds: if the
induced involution on $L_K(E)$ is positive definite for every (equivalently some) graph $E$, then the
involution on $K$ is positive definite (Proposition \ref{moreimplications}).

In \cite{AR}, Gene Abrams and the second named author characterized the (von Neumann) regular Leavitt path
algebras $L_K(E)$ as precisely those with acyclic underlying graphs $E.$ In light of our consideration of the
involution on $L_K(E),$ we wonder when is $L_K(E)$  $^*$-regular (regular with proper involution). In Theorem
\ref{characterizationtheorem}, we characterize $^*$-regular Leavitt path algebras as exactly those with $E$
acyclic and $K$ that is proper up to a certain extent (determined by the least upper bound of the numbers of
all paths that end at any given vertex of $E$). Note that we do not impose any conditions on the cardinality
of $E$: we work with completely arbitrary graphs.

Most existing characterization theorems for Leavitt path algebras have the following form:
\begin{center}
{\em $L_K(E)$ has (ring-theoretic) property $(P)$ if and only if  $E$ has (graph-theoretic) property $(P').$ }
\end{center}
Such theorems have been formulated and proven for a good number of ring-theoretic properties. For example
simple, purely infinite simple, exchange, semisimple, regular and other Leavitt path algebras have been
characterized. It is interesting that the underlying field $K$ did not play any role in those characterization
theorems. Theorem \ref{characterizationtheorem}, however, has a different form:
\begin{center}
{\em $L_K(E)$ has property $(P)$ if and only if  $E$ has property $(P')$ {\em and} $K$ has property $(P'').$
}
\end{center}
In other words, Theorem \ref{characterizationtheorem} is the first characterization theorem that involves a
ring-theoretic property of the field $K$ as well. Moreover, the characterization of $L_K(E)$ that is positive
definite (Proposition \ref{moreimplications}) also has the above form that features the field $K$ as well.

The paper is organized as follows. In \S 1 we recall the basic definitions, examples and properties of Leavitt
path algebras, whereas in \S 2 we focus on the involution of $L_K(E)$ and prove Propositions
\ref{involutionproperarbitrarygraphs} and \ref{moreimplications}. We devote \S 3 to the proof of the
characterization theorem for the $^\ast$-regular Leavitt path algebras (Theorem
\ref{characterizationtheorem}). Finally, in \S 4 we consider \cite[Problem 48, p. 380]{G} that we shall refer
to as ``Handelman's conjecture''. This conjecture is stating that every $^\ast$-regular ring is unit-regular.
We prove that Handelman's conjecture holds for Leavitt path algebras. In fact, we formulate a generalized
version of the conjecture for rings with local units and show that it holds for Leavitt path algebras over
arbitrary graphs.

\section{Definitions and preliminaries}

We recall some graph-theoretic concepts, the definition and standard examples of Leavitt path algebras.

A (\emph{directed}) \emph{graph} $E=(E^0,E^1,r,s)$ consists of two sets $E^0$ and $E^1$ (with no restriction
on their cardinals) together with maps $r,s:E^1 \to E^0$. The elements of $E^0$ are called \emph{vertices} and
the elements of $E^1$ \emph{edges}. For $e\in E^1$, the vertices $s(e)$ and $r(e)$ are called the
\emph{source} and \emph{range} of $e$. If $s^{-1}(v)$ is a finite set for every $v\in E^0$, then the graph is
called \emph{row-finite}. If $E^0$ is finite and $E$ is row-finite, then $E^1$ must necessarily be finite as
well; in this case we say simply that $E$ is \emph{finite}.

A vertex which emits (receives) no edges is called a \emph{sink (source)}. A vertex $v$ is called an
\emph{infinite emitter} if $s^{-1}(v)$ is an infinite set. A \emph{path} $\mu$ in a graph $E$ is a finite
sequence of edges $\mu=e_1\dots e_n$ such that $r(e_i)=s(e_{i+1})$ for $i=1,\dots,n-1$. In this case,
$s(\mu)=s(e_1)$ and $r(\mu)=r(e_n)$ are the \emph{source} and \emph{range} of $\mu$, respectively, and $n$ is
the \emph{length} of $\mu$. We view the elements of $E^{0}$ as paths of length $0$.

\smallskip

If $\mu$ is a path in $E$, with $v=s(\mu)=r(\mu)$ and $s(e_i)\neq s(e_j)$ for every $i\neq j$, then $\mu$ is
called a \emph{cycle}. A graph which contains no cycles is called \emph{acyclic}.

Let $K$ denote an arbitrary base field and $E$ an arbitrary graph. The {\em Leavitt path $K$-algebra} $L_K(E)$
is the $K$-algebra generated by the set $E^0\cup E^1\cup \{e^*\mid e\in E^1\}$ with the following relations:
\begin{enumerate}
\item[(V)] $vw= \delta_{v,w}v$ for all $v,w\in E^0$.
\item[(P1)] $s(e)e=er(e)=e$ for all $e\in E^1$.
\item[(P2)] $r(e)e^*=e^*s(e)=e^*$ for all $e\in E^1$.
\item[(CK1)] $e^*f=\delta _{e,f}r(e)$ for all $e,f\in E^1$.
\item[(CK2)] $v=\sum _{e\in s^{-1}(v)}ee^*$ for every $v\in E^0$ that is neither a sink nor an infinite emitter.
 \end{enumerate}

The first three relations are the path algebra relations. The last two are the so-called Cuntz-Krieger
relations.

We let $r(e^*)$ denote $s(e)$, and we let $s(e^*)$ denote $r(e)$.  If $\mu = e_1 \dots e_n$ is a path in $E$,
we write $\mu^*$ for the element $e_n^* \dots e_1^*$ of $L_{K}(E)$. With this notation, the Leavitt path
algebra $L_{K}(E)$ can be viewed as a $K$-vector space span of $\{pq^{\ast } \ \vert \ p,q\,  \hbox{are paths
in} \,  E\}$. (Recall that the elements of $E^0$ are viewed as paths of length $0$, so that this set includes
elements of the form $v$ with $v\in E^0$.)

If $E$ is a finite graph, then $L_{K}(E)$ is unital with $\sum_{v\in E^0} v=1_{L_{K}(E)}$; otherwise,
$L_{K}(E)$ is a ring with a set of local units consisting of sums of distinct vertices of the graph.

Many well-known algebras can be realized as the Leavitt path algebra of a graph. The most basic graph
configurations are shown below (the isomorphisms for the first three can be found in \cite{AA1}, the fourth in
\cite{S}, and the last one in \cite{AA2}).

\begin{examples}\label{examples}{\rm The ring of Laurent polynomials $K[x,x^{-1}]$ is the Leavitt path
algebra of the graph given by a single loop graph. Matrix algebras ${\mathbb M}_n(K)$ can be realized by the
line graph with $n$ vertices and $n-1$ edges. Classical Leavitt algebras $L(1,n)$ for $n\geq 2$ can be
obtained by the $n$-rose -- a graph with a single vertex and $n$ loops. Namely, these three graphs are:

$$
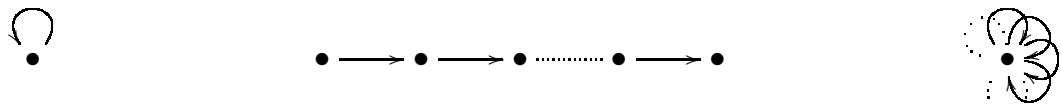
\begin{matrix} \xymatrix{{\bullet} \ar@(ur,ul)} \hskip3cm &
\xymatrix{{\bullet} \ar [r]  & {\bullet} \ar [r]  & {\bullet} \ar@{.}[r] & {\bullet} \ar [r]  & {\bullet} }
\hskip3cm  & \xymatrix{{\bullet} \ar@(ur,dr)  \ar@(u,r)  \ar@(ul,ur)  \ar@{.} @(l,u) \ar@{.} @(dr,dl)
\ar@(r,d) & }
\end{matrix}$$

\medskip

The algebraic counterpart of the Toeplitz algebra $T$ is the Leavitt path
algebra of the graph having one loop and one exit:
$$
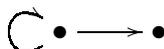
\xymatrix{{\bullet} \ar@(dl,ul) \ar[r] & {\bullet}  }$$

Combinations of the previous examples are possible. For instance, the Leavitt path algebra of the graph

$$
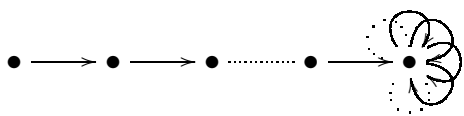
\xymatrix{{\bullet} \ar [r]  & {\bullet} \ar [r]  & {\bullet}
\ar@{.}[r] & {\bullet} \ar [r]  & {\bullet}
 \ar@(ur,dr)  \ar@(u,r)  \ar@(ul,ur)  \ar@{.} @(l,u) \ar@{.} @(dr,dl) \ar@(r,d) & }$$ \smallskip

\noindent is ${\mathbb M}_n(L(1,m))$, where $n$ denotes the number of vertices in the graph and $m$ denotes
the number of loops.
}
\end{examples}

\section{The involution on a Leavitt path algebra}

We recall some standard definitions first.

An \emph{involution} $^\ast$ on a ring $R$ is an additive map $^\ast:R\rightarrow R$ that satisfies
$(ab)^{\ast }=b^{\ast}a^{\ast }$ and $(a^{\ast })^{\ast }=a$ for all $a,b\in R.$ For any $a\in R$, the element
$a^{\ast }$ is called the \emph{adjoint} of $a$. An element $p$ in a ring with involution $^\ast$ is called a
\emph{projection} if $p$ is a self-adjoint $(p^{\ast }=p)$ idempotent $(p^{2}=p)$.

If there is an involution defined on a ring $R,$ $R$ is said to be a {\em $^\ast$-ring}. If $R$ is also an
algebra over $K$ with an involution $^-,$ then $R$ is a {\em $^\ast$-algebra} if $(ax)^*=\overline{a}x^*$ for
$a\in K,$ and $x\in R.$

Let $n$ be a positive integer. An involution $^\ast$ on a ring $R$ is said to
be \emph{n-proper} if
$$x_{1}^{\ast}x_{1}+\dots+x_{n}^{\ast}x_{n}=0 \text{ implies } x_{1}=\dots =x_{n}=0$$ for any $n$
elements $x_{1},\dots,x_{n}$ in $R$. A ring with an $n$-proper involution will be refered to as an
$n$-\emph{proper} ring. This property is clearly left-right symmetric, since each element $x_{i}=a_{i}^{\ast}$
for some $a_{i}\in R$. 1-proper involution is simply said to be {\em proper}.

The involution $^\ast$ is said to be \emph{positive definite} if it is $n$-proper for every positive integer
$n.$ A $^\ast$-ring with a positive definite involution will be referred to as a {\em positive definite} ring.

A field can have both an $n$-proper involution and an involution that is not $n$-proper. For example, consider
the field $\mathbb{C}$: it is 2-proper (in fact positive definite) for the conjugate involution ($a+ib\mapsto
a-ib$) and not 2-proper for the identity involution.  Also, the same involution can be $n$-proper and not
$n+1$-proper (identity involution on $\mathbb{C}$ for $n=1$).

Before we turn to Leavitt path algebras, let us recall one last fact about general $^\ast$-rings. Recall that
if $R$ is a ring with involution $^-$, then the involution $^-$ induces the $^\ast$-\emph{transpose
involution} on the ring ${\mathbb M}_{n}(R)$ of $n\times n $ matrices over $R$ given by $$A=(a_{ij})\mapsto
A^*=(\overline{a_{ji}}).$$

We believe that the following lemma is well known but we are not aware of a reference for it. For
completeness, we provide a proof here.

\begin{lemma}\label{R_n-proper}
Let $n$ be a positive integer and let $R$ be a ring with involution $^-$. Then the $^\ast$-transpose
involution on ${\mathbb M}_{n}(R)$ is proper if and only if the involution $^-$ in $R$ is $n$-proper.
\end{lemma}
\begin{proof}
Assume that the involution $^-$ is $n$-proper in $R$. Suppose $A^*A=0$ for some matrix $A=(a_{ij})\in {\mathbb
M}_{n}(R)$.  Then the diagonal entries of the product $A^*A$ are zero and so
$\sum_{j=1}^{n}\overline{a_{ji}}a_{ji}=0$ for every $i=1,\dots,n$. Since $^-$ is $n$-proper, $a_{ij}=0$ for
all $i,j$. Hence $A=0.$

Conversely, suppose $\sum_{i=1}^{n}\overline{a_{i}}a_{i}=0$ for $a_{i}\in R $. Consider $A$ to be the matrix
of ${\mathbb M}_{n}(R)$ that has the elements $a_{1},\dots,a_{n}$ in its first column and zeroes in the rest
of its columns. Then $A^*A=0$. Since the $^\ast$-transpose involution is proper,
$A=0.$ So $a_{i}=0$ for every $i$.
\end{proof}

We turn now to Leavitt path algebras. Let $K$ be a field with involution $^-$ and let $E$ be an arbitrary
graph. Recall that a typical element of the Leavitt path algebra $L_{K}(E)$ can be written as
$\sum_{i=1}^{n}k_{i}p_{i}q_{i}^{\ast}$ where $p_{i}$ and $q_{i}$ are paths and $k_i\in K.$ It is
straightforward to see that the map $^\ast$ given by
$$\left(\sum_{i=1}^{n}k_{i}p_{i}q_{i}^{\ast}\right)^{\ast}=
\sum_{i=1}^{n}\overline{k_i}q_{i}p_{i}^{\ast}$$
defines the involution on $L_{K}(E)$ making it into a $^\ast$-algebra.

If $K$ is the field of complex numbers $\mathbb{C}$ and we consider the conjugate involution, the
$^\ast$-algebra structure of $L_{\mathbb{C}}(E)$ agrees with the $^{\ast}$-algebra structure of
$L_{\mathbb{C}}(E)$ (that is used for instance to see $L_{\mathbb{C}}(E)$ as a dense $^*$-subalgebra of
$C^*(E)$ as shown in \cite[Theorem 7.3]{Tomforde}).

As we will see in the next section, the $^*$-regularity of a Leavitt path algebra $L_K(E)$ is closely related
to the condition stating that the involution of $L_K(E)$ is $n$-proper or positive definite. In this section,
we characterize the positive definiteness of $L_K(E)$.

The following proposition can be proved by easily adapting Iain Raeburn's result \cite[Lemma 1.3.1]{Iain} to
our notation and context.

\begin{proposition}\label{Iain} Let $E$ be a row-finite, countable graph without sinks. If the involution
$^-$ on $K$ is positive definite, then the involution $^\ast$ on $L_K(E)$ is proper.
\end{proposition}
\begin{proof}
The proof follows completely \cite[Lemma 1.3.1]{Iain}. One only needs to take into account that the axioms in
\cite{Iain} are given so that a path $e_1e_2\ldots e_n$ in  \cite{Iain} corresponds to the path
$e_ne_{n-1}\ldots e_1$ here (i.e. the edges in \cite{Iain} are multiplied so that the action of $f$ precedes
the action of $e$ in the product $ef,$ contrary to the action we consider here). Because of this difference,
the assumptions of \cite[Lemma 1.3.1]{Iain} that $L_K(E)$ is column-finite with no sources, correspond exactly
to our assumptions that $L_K(E)$ is row-finite with no sinks.
\end{proof}

It is noted in \cite{Iain} that the condition that $E$ does not have sinks (sources, in the terminology of
that paper) can be avoided by using the so-called Yeend's trick. This assumption can also be avoided using a
technique called the Desingularization Process. The added benefit of the desingularization is that it can help
us also get rid of the row-finiteness assumption. The desingularization of a graph $E$ is a new graph $F$
obtained by adding a tail (more details can be found in  \cite{DT} or \cite{AA3}) to every sink or infinite
emitter. The resulting graph $F$ is a row-finite graph without sinks such that $L_K(E)$ embeds in $L_K(F)$ via
the embedding that is a $^*$-algebra homomorphism (see \cite[Proposition 5.1]{AA3} for more details).

Finally, the last remaining assumption (that the graph is countable) in Proposition \ref{Iain} can be avoided
by means of the Subalgebra Construction (see \cite{AR} for more details). We recall here the relevant concept
$E_F$ used in this construction. We shall use $E_F$ in our main theorem too.

Let $F$ be a finite set of edges in $E$. We define $s(F)$ (resp. $r(F)$) to be the sets of those vertices in
$E$ that appear as the source (resp. range) vertex of at least one element of $F$. The graph $E_F$ is then
defined as follows (see \cite[Definition 2]{AR}):

\vskip-0.2cm

$$E_F^0=F\cup (r(F)\cap s(F)\cap s(E^1\setminus F))\cup (r(F)\setminus s(F)),$$

\vskip-0.5cm

$$E_F^1=\{(e,f)\in F\times E^0_F \ | \ r(e)=s(f)\}\cup [\{(e,r(e))\ | \ e\in F \text{ with } r(e)\in
(r(F)\setminus s(F))\}],$$

\vskip-0.5cm

$$\text{and where }s((x,y))=x \text{ and } r((x,y))=y\text{ for any }(x,y)\in E_F^1.$$

The graph $E_F$ is finite (see comment after \cite[Definition 2]{AR}). Also, by \cite[Proposition 1]{AR},
there is an algebra homomorphism $\theta: L_K(E_F)\rightarrow L_K(E).$ Further, \cite[Proposition 2]{AR} shows
that for every finite set of elements $S$ of $L_K(E),$ there is a subalgebra $B(S)$ of $L_K(E)$ containing
$S.$ The subalgebra $B(S)$ is of the form $L_{K}(E_{F})\bigoplus (\bigoplus_{i=1}^m Kx_{i})$ where $F$ is a
finite set of edges defined using $S$ (see \cite[page 7]{AR}) and $x_i, i=1,\ldots, m$ is a finite set of
vertices (defined in \cite[page 8]{AR}). By  \cite[Proposition 2]{AR}, $L_{K}(E)$ is a directed union of
subalgebras $B(S)$, where the $S$ varies over all finite subsets of $L_K(E)$. Furthermore, and key to our
current discussion, $\theta$ preserves the involution by the construction (as it can be seen from the proof of
\cite[Proposition 1]{AR}) so it is a $^\ast$-algebra homomorphism.

\begin{proposition}\label{involutionproperarbitrarygraphs} Let $E$ be an arbitrary graph. If the involution
$^-$ on $K$ is positive definite, then the involution $^*$ on $L_{K}(E)$ is proper.
\end{proposition}
\begin{proof}

We prove the claim first for the case when $E$ is a countable. Let $F$ be a desingularization of $E.$ By the
Desingularization Process, $F$ is a row-finite graph without sinks and there is a $^\ast$-algebra monomorphism
$\phi:L_{K}(E)\to L_{K}(F).$ Now, suppose that $a^*a=0$ in $L_{K}(E)$. Apply $\phi$ to get that
$\phi(a)^*\phi(a)=0$ in $L_{K}(F).$ Since $F$ is row-finite and does not contain sinks, Proposition \ref{Iain}
can be applied and so $\phi(a)=0.$ Then $a=0$ since $\phi$ is a monomorphism.

\par
Now suppose that $E$ is arbitrary and let $a\in L_{K}(E)$ be such that $a^*a=0$. By the Subalgebra
Construction, for the finite set $S=\{a, a^*\},$ there is a finite set of edges $F$ and a finite number of
vertices $x_i, i=1,\ldots, m$ such that the subalgebra $B(S)$ of $L_K(E)$ is of the form
$L_{K}(E_{F})\bigoplus (\bigoplus_{i=1}^m Kx_{i})$ and $a,a^*\in B(S)$. Since $\bigoplus_{i=1}^m Kx_{i}$ is a
direct summand in the previous equation for $B(S)$, we can actually attach a finite number of isolated
vertices $v_{1},\dots,v_{m}\not\in E_{F}^0$ to the graph $E_{F}$  so that we obtain a new finite graph $G$
such that
$$L_K(G)\cong L_{K}(E_{F})\bigoplus \left(\bigoplus_{i=1}^m Kv_{i}\right)$$ via a $^*$-algebra isomorphism.

Since $B(S)$ is a subalgebra of $L_{K}(E)$, the equation $a^*a=0$ holds in $B(S)\cong L_{K}(G)$. Apply the
previous case to $L_{K}(G)$ in order to deal with possible sinks in $G.$ Then we have that $a=0$. This
finishes the proof.
\end{proof}

\medskip

The last result of this section is a characterization of Leavitt path algebras that have positive definite
involutions in terms of the corresponding property in the field $K$.

\begin{proposition}\label{moreimplications} Let $K$ a field with involution. The following conditions
are equivalent.
\begin{enumerate} [{\rm (i)}]
\item The involution on $K$ is positive definite.
\item The involution on $L_{K}(E)$ is positive definite for every graph $E$.
\item The involution on $L_{K}(E)$ is positive definite for some graph $E$.
\end{enumerate}
Thus, if $E$ is an arbitrary graph, $L_{K}(E)$ is positive definite if and only if $K$ is positive definite.
\end{proposition}
\begin{proof}
(i) $\Longrightarrow$ (ii). Given $E,$ let us consider the graph $M_nE$ obtained from $E$ by attaching a line
of length $n-1$ to every vertex of $E$ so that each line ends at the given vertex of the graph (more details
in \cite{AT}). The graph $M_nE$ has the property that ${\mathbb M}_n(L_K(E))$ is isomorphic to $L_K(M_nE)$ as
$^\ast$-algebras by \cite[Proposition 9.3]{AT}.

By Proposition \ref{involutionproperarbitrarygraphs}, we know that $L_K(M_nE)$ is $^*$-proper. So, ${\mathbb
M}_n(L_K(E))$ is $^*$-proper. Now apply Lemma \ref{R_n-proper} to get that $L_K(E)$ is $n$-proper.

\par

(ii) $\Longrightarrow$ (iii) is a tautology.

\par

(iii) $\Longrightarrow$ (i). Suppose that $\sum_{i=1}^n \overline{k_{i}}k_{i}=0$ for $k_{i}\in K$. Let $E$ be
a graph such that the involution on $L_{K}(E)$ is positive definite. Let $v\in E^0$. Since $v$ is a projection
then $0=(\sum_{i=1}^n \overline{k_{i}}k_{i})v=\sum_{i=1}^n(k_{i}v)^* (k_{i}v) $ and therefore $k_{i}v=0$ for
all $i$ by hypothesis. But $E^0$ is a set of linearly independent elements in $L_{K}(E)$ by \cite[Lemma
1.5]{G}, so that $k_{i}=0$ for all $i$, as needed.
\end{proof}

\section{$^*$-regular Leavitt path algebras}

The (von Neumann) regular Leavitt path algebras $L_K(E)$ were characterized in \cite{AR} as those whose graphs
$E$ are acyclic. In light of the consideration of $L_K(E)$ as a ring with involution, we wonder which acyclic
graphs have $L_K(E)$ that is $^*$-regular. We provide an answer to this question in this section.

Recall that a ring $R$ is (von Neumann) regular if for every $a\in R$ there exists $b\in R$ such that $aba=a$,
or equivalently \cite[Theorem 4.23]{Lam}, every right (resp. left) principal ideal is generated by an
idempotent. This statement continues to hold if $R$ is a ring with local units since $b\in bR$ (and $b\in Rb$)
for all $b$ in $R$ so the principal right (and left) ideals of $R$ have the same form as the principal right
(left) ideals of a unital ring.

If $R$ is a $^\ast$-ring, the projections take over the role of idempotents. Thus, the concept of regularity
for rings corresponds to $^\ast$-regularity for $^*$-rings: a $^\ast$-ring $R$ is said to be
$^\ast$\emph{-regular} if every principal right ideal is generated by a projection. This definition naturally
extends to rings with local units. Note that the condition of $^\ast$-regularity is left-right symmetric since
$aR=pR$ implies that $Ra^{\ast}=Rp$ for any $a\in R$ and a projection $p\in R.$ So every principal left ideal
of $R$ is also generated by a projection in the case when every principal right ideal is.

A $^\ast$-ring is $^\ast$-regular if and only if it is regular and the involution $^*$ is proper (see
\cite[Exercise 6A, \S 3]{Be}). In the next proposition we give a proof of this fact for rings with local
units.

\begin{proposition}\label{regularproper}
Let $R$ be a ring with local units and with an involution $^\ast$. Then $R$ is $^\ast$-regular if and only if
$R$ is regular and $^\ast$ is proper.
\end{proposition}
\begin{proof}
If $R$ is $^\ast$-regular then it is also regular because every projection is an idempotent. Now assume that
$a^{\ast}a=0$ for some $a\in R.$ Then $aR=pR$ for some projection $p$ so $a=pa$ ($a=px$ for some $x\in R$
implies that $pa=ppx=px=a$) and $p=ay$ for some $y\in R.$ So, $a^{\ast}=a^{\ast}p=a^{\ast}ay=0$. Hence $a=0.$
Thus $^\ast$ is proper.

Conversely, suppose that $R$ is regular and $^\ast$ is proper. Since every principal right ideal is generated
by an idempotent, it is enough to show that for an arbitrary idempotent $x$ in $R$, $Rx^{\ast}=Rp $ for some
projection $p\in R$. First observe that for any $x\in R$, $r_{R}(x)=r_{R}(x^{\ast}x)$ where $r_{R}(b)$ denotes
the right annihilator of the element $b\in R$. This is because, $x^{\ast}xy=0$ implies that
$(xy)^{\ast}xy=y^{\ast}(x^{\ast}xy)=0$ so that $xy=0$ for any $y\in R$. By the regularity of $R$,
$Rx^{\ast}x=Rf$ for some idempotent $f\in R$. Thus $r_{R}(x)=r_{R}(x^{\ast}x)=r_{R}(f)$ and so the left
annihilators $l_{R}(r_{R}(x))$ and $l_{R}(r_{R}(f))$ are also equal.

We claim that $Rx=l_{R}(r_{R}(x))$. To see this, first note that, since $x$ is an idempotent,
$r_{R}(x)=\{a-xa\ | \ a\in R\} $. So if $y\in l_{R}(r_{R}(x))$, then $y(a-xa)=0$ for all $a\in R$, that is
$ya=yxa$ for all $a\in R$. Since $R$ is a ring with local units, there is an idempotent $u\in R$ such that
$yu=y$ and $xu=x$. Hence $y=yu=yxu=yx\in Rx$. Thus $l_{R}(r_{R}(x))\subseteq Rx$. Since the reverse inclusion
is obvious, $Rx=l_{R}(r_{R}(x))$.

Similarly, $Rf=l_{R}(r_{R}(f))$. Thus $Rx=Rf=Rx^{\ast}x $. Hence $x=ax^{\ast}x$ for some $a\in R$. Let
$p=ax^{\ast}$. We claim that $p$ is a projection with $Rx^{\ast}=Rp$. To see this, note that $x=px$ and so
$pp^{\ast}=pxa^{\ast}=xa^{\ast}=p^{\ast}$. Since $(pp^{\ast})^{\ast}=pp^{\ast}$, we get $p=p^{\ast}$. From
$pp^*=p^*$ and $p=p^{\ast}$ we have $p^{2}=p$ and so $p$ is a projection. From $p=ax^{\ast}$ and
$x^*=x^{\ast}p$ we have that $Rx^{\ast}=Rp.$
\end{proof}

We turn to Leavitt path algebras now. For any vertex $v$ in a graph $E$, let $\mu (v)$ denote the cardinality
of the set of all the paths $\alpha$ in $E$ with $r(\alpha )=v$ (including the trivial path $v$). With this
notation, we recall the statement of \cite[Lemma 3.4 and Proposition 3.5]{AAS1}. Let $E$ be a finite acyclic
graph and $v$ a sink in $E.$ The set $$I_{v}=\left\{\sum_{i}k_{i}\alpha _{i}\beta _{i}^{\ast}\ | \ \alpha
_{i},\beta _{i}\text{ paths in }E\text{ with }r(\alpha _{i})=r(\beta _{i})=v,k_{i}\in K\right\}$$ is an ideal
of $L_{K}(E)$ isomorphic to the matrix ring ${\mathbb M}_{\mu(v)}(K).$ If $\{v_{1},\dots,v_{m}\}$ is the set
of all sinks in $E,$ then $L_{K}(E)=\bigoplus _{i=1}^{m}I_{v_{i}}\cong \bigoplus _{i=1}^{m}{\mathbb M}_{\mu
(v_{i})}(K)$. Let us denote this isomorphism by $\phi$ and let us call it a \emph{canonical isomorphism}.

From \cite[Lemma 3.4 and Proposition 3.5]{AAS1} it can be seen that the restriction of $\phi$ on a direct
summand $I_v,$ for a vertex $v,$ is given by $\phi(\sum_{i,j}k_{ij}\alpha _{i}\alpha _{j}^{\ast})=(k_{ij})\in
{\mathbb M}_{\mu (v)}(K)$ where $i,j=1,\dots,\mu(v),$ $\alpha_i$ and $\alpha_j$ are paths ending at $v$ and
$k_{ij}\in K.$

\begin{lemma}\label{canonicalisomorphism} Let $E$ be a finite acyclic graph and let $\{v_{1},\dots,v_{m}\}$
be all the sinks in $E$. The canonical isomorphism $\phi:L_{K}(E)=\bigoplus _{i=1}^{m}I_{v_{i}}\rightarrow
\bigoplus _{i=1}^{m}{\mathbb M}_{\mu (v_{i})}(K)$ is a $^*$-algebra isomorphism (with the standard involution
on $L_{K}(E)$ and the $^\ast$-transpose involution on the matrix algebras).
\end{lemma}
\begin{proof}
Since $\phi$ maps direct summands $I_{v_i}$ on direct summands ${\mathbb M}_{\mu (v_{i})}(K),$ it is enough if
we prove the statement when $E$ has only one sink $v$. If $\alpha _{1},\dots,\alpha _{\mu (v)}$ are all the
different paths (including the trivial path) ending in $v,$ then a typical element of $L_K(E)=I_{v}$ has the
form $\sum_{i,j}k_{ij}\alpha _{i}\alpha _{j}^{\ast}$ for  $i,j=1,\dots,\mu(v),$ $\alpha_i$ and $\alpha_j$
paths ending at $v$ and $k_{ij}\in K.$ Then we have
$$\phi \left(\left(\sum_{i,j}k_{ij}\alpha _{i}\alpha_{j}^{\ast}\right)^{\ast}\right)=
\phi \left(\sum_{i,j}\overline{k_{ij}}\alpha _{j}\alpha_{i}^{\ast}\right)=
\phi \left(\sum_{i,j}\overline{k_{ji}}\alpha _{i}\alpha_{j}^{\ast}\right)
=(\overline{k_{ji}})=(k_{ij})^*.$$
This proves the claim since $(k_{ij})^*=\left(\phi \left(\sum_{i,j}k_{ij}\alpha _{i}\alpha_{j}^{\ast}\right)
\right)^{\ast}.$
\end{proof}

\medskip

We finally have all the ingredients in hand to prove the main result of the paper.

\begin{theorem}\label{characterizationtheorem}
Let $E$ be an arbitrary graph, $K$ be a field with involution $^-$ and let $\sigma=\sup \{\mu (v):v\in
E^{0}\}$ in case the supremum is finite or $\sigma=\omega$ otherwise. The following conditions are equivalent.
\begin{enumerate}[{\rm (i)}]
\item $L_{K}(E)$ is $^\ast$-regular.
\item $L_{K}(E)$ is regular and proper.
\item $E$ is acyclic and $K$ is $n$-proper for every finite $n\leq \sigma$.
\end{enumerate}
\end{theorem}
\begin{proof}
(i) $\Leftrightarrow$ (ii) is Proposition \ref{regularproper}.

\medskip

(ii) $\Leftrightarrow$ (iii). By \cite[Theorem 1]{AR}, $L_{K}(E)$ is  regular if and only if $E$ is acyclic.
So it is enough if we show, under the assumption that $L_{K}(E)$ is regular (equivalently, $E$ is acyclic),
that the involution $^-$ in $K$ is $n$-proper for every finite $n\leq \sigma$ if and only if the involution
$^\ast$ in $L_{K}(E)$ is proper.

Now \cite[Proposition 2 and Theorem 1]{AR} also state that, when $E$ is acyclic, $L_{K}(E)$ is a directed
union of subalgebras $B(S)$ where each $B(S)\overset{\theta }{\cong }L_{K}(E_{F})\bigoplus (\bigoplus
_{i=1}^{m}Kx_{i})$ with $E_{F}$ a finite acyclic graph constructed corresponding to various non-empty finite
subsets $F$ of edges in $E$. Moreover $\theta$ is a $^\ast$-algebra isomorphism as we noted before. For a
fixed $F$, $E_F$ has a finite number of sinks. Let us denote them by $v_{1},\dots, v_k$. Then
$L_{K}(E_{F})\overset{\phi }{\cong }\bigoplus _{i=1}^{k}{\mathbb M}_{\mu _{E_{F}}(v_{i})}(K)$ as
$^\ast$-algebras by Lemma \ref{canonicalisomorphism}. Thus, the involution $^\ast$ in $L_{K}(E)$ is proper if
and only if the $^\ast$-transpose involution is proper in ${\mathbb M}_{\mu _{^{E_{F}}}(v_{i})}(K)$ with
$v_{i}\in E_{F}$ for all the various graphs $E_{F}$ corresponding to each $B(S)$ in the stated directed system
of subalgebras of $L_{K}(E)$.

We distinguish two situations.

\smallskip

\underline{Case 1}: Suppose $\sigma $ is infinite. Then either $\mu (v)$ is infinite for some vertex $v$ or
for every positive integer $n$ there is a vertex $v_{n}$ with $\mu _{E}(v_{n})$ an integer larger than $n$. In
either case for each integer $n>1,$ we can choose a vertex $v_{n}$ and a finite subset $F_{n}$ of edges that
appear in the $n-1$ distinct paths (other than the trivial path $v_{n}$) ending in $v_{n}$. The vertex $v_{n}$
is a sink in $E_{F_{n}}$ by the definition of the graph $E_{F_{n}}.$

Moreover, $e_{1}\dots e_{k}$ is a path of length $k$ in $E$ ending in $v_{n}$ if and only if the path in
$E_{F_n}$ given by $(e_{1},e_{2})(e_{2},e_{3})\dots (e_{k},v_{n})$  has length $k$ and ends in $v_{n}$. Thus,
$v_{n}$ is a sink in $E_{F_{n}}$ with $\mu_{E_{F_{n}}}(v_{n})=n.$ The graph $E_{F_{n}}$ is finite acyclic (by
\cite[Lemma 1]{AR}). So, $L_{K}(E_{F_{n}})$ contains the ideal $I_{v_{n}}\cong {\mathbb M}_{n}(K)$.

Since this holds for every $n$, the involution $^\ast$ is proper in each subalgebra $B(S)$ if and only if the
$^\ast$-transpose involution is proper in ${\mathbb M}_{n}(K)$ for each positive integer $n$. This is
equivalent, by Lemma \ref{R_n-proper}, to the statement that the involution $^-$ in $K$ is $n$-proper for
every $n$, that is, that $K$ is positive definite.

\smallskip

\underline{Case 2}: Suppose that $\sigma$ is finite, say $\sigma =n$ for some positive integer $n$. If $n=1$,
every vertex in $E$ is isolated and $L_{K}(E)$ is isomorphic to $\bigoplus _{v\in E^{0}}Kv$ where $Kv\cong K.$
Both of those algebras are proper so we are done.

Suppose $n>1$ and let $v$ be a vertex for which $\mu (v)=n$.
Let $F_{v}$ be the non-empty finite
set of edges in all the $\mu _{E}(v)-1$ nontrivial paths ending in $v$. As noted in Case 1, $v$ is a sink (and
in this case, the only sink) in the finite acyclic graph $E_{F_{v}}$ and, moreover, $\mu
_{E_{F_{v}}}(v)=\mu_{E}(v)$.

So by Lemma \ref{canonicalisomorphism}, we have that $L_{K}(E_{F_{v}})$ $\cong {\mathbb M}_{\mu _{E}(v)}(K)$
as $^*$-algebras. Moreover, as $n$ is the least upper bound of $\{\mu (v):v\in E^{0}\}$ then all matrices
${\mathbb M}_{\mu_{E_{F_i}}(v_i)}(K)$ appearing in all other subalgebras $B(S)$ for various other finite
subsets of edges $F_i$ and sinks $v_i$ will all have order that is less or equal to $n$.

Therefore the involution $^\ast$ is proper in each $B(S)$ if and only if the $^\ast$-transpose involution is
proper in ${\mathbb M}_{\mu _{E}(v)}(K)$. Since $\mu (v)=n$, the last statement holds exactly when the
involution $^-$ in $K$ is $n$-proper, again by Lemma \ref{R_n-proper}. This finishes the proof.
\end{proof}

It is interesting to point out that the presence of involution gives a more prominent role to the field $K$
than it had in the previous characterization theorems (e.g., simplicity \cite{AA1}, purely infinite simplicity
\cite{AA2}, finite-dimensionality \cite{AAS1}, just to cite a few). In particular, Theorem
\ref{characterizationtheorem} also contrasts the characterization of regularity from \cite{AR} that was
independent of the field $K.$

We further illustrate this behavior with an easy example. If $E$ is the graph
$$\xymatrix{{\bullet} \ar [r]  & {\bullet} }$$
then $L_K(E)\cong {\mathbb M}_2(K)$ as $^\ast$-algebras for any field $K$. If $K={\mathbb R}$ with the
identity involution, $L_{\mathbb{R}}(E)$ is $^*$-regular because ${\mathbb R}$ is positive definite. However
if $K={\mathbb C}$ with the identity involution, then $L_{\mathbb{C}}(E)\cong {\mathbb M}_{2}( \mathbb{C})$ is
regular but it is \emph{not} $^\ast$-regular (since the identity involution in $\mathbb{C}$ is not
$2$-proper). Furthermore, if $K={\mathbb C}$ with the conjugate involution, then $L_{\mathbb{C}}(E)$ \emph{is}
$^*$-regular, because the conjugation of complex numbers is positive definite.

Also, since the identity involution on a field of characteristic $n>0$ is not $n$-proper, the properness (thus
also $^*$-regularity) of a Leavitt path algebra over such field depends on the characteristic of the field.
This fact also brings the field characteristic into spotlight.

Let us note the following corollary of Theorem \ref{characterizationtheorem}.

\begin{corollary}\label{starregularequivalences} Let $K$ a field with involution. The following conditions
are equivalent.
\begin{enumerate} [{\rm (i)}]
\item The involution on $K$ is positive definite.
\item $L_{K}(E)$ is $^\ast$-regular for every acyclic graph $E$.
\end{enumerate}
\end{corollary}
\begin{proof}
(i) $\Longrightarrow$ (ii) follows directly from Theorem \ref{characterizationtheorem}.

\par

(ii) $\Longrightarrow$ (i) Let us assume that $L_K(E)$ is $^\ast$-regular for every acyclic $E$. Consider the
line of length $n-1$ (see second graph in Examples \ref{examples}). The Leavitt path algebra of this graph is
isomorphic to ${\mathbb M}_{n}(K).$ From the assumption that this algebra is $^*$-regular, we obtain that $K$
is $n$-proper by Lemma \ref{R_n-proper}. Since this holds for every $n$, $K$ is positive definite.
\end{proof}

It is also interesting to note that the two equivalences of Corollary \ref{starregularequivalences} parallel
the first two equivalences of Proposition \ref{moreimplications}. The last equivalence of Proposition
\ref{moreimplications} in the $^*$-setting would have the form: `` $L_{K}(E)$ is $^\ast$-regular for some
acyclic graph $E$''. However, this statement is weaker than the other two equivalences in Corollary
\ref{starregularequivalences} so we do not have complete analogy with Proposition \ref{moreimplications}. To
see this, consider a graph consisting of a single vertex and the complex numbers with the identity involution.
The Leavitt path algebra of this graph is $^\ast$-regular but the field is not positive definite.

\section{Handelman's Conjecture for Leavitt path algebras}

We close this paper by pointing out that Handelman's conjecture has a positive answer for the family of
Leavitt path algebras of arbitrary graphs. The conjecture can be stated as follows.

\begin{conjecture} {\rm (Handelman, \cite[Problem 48, p. 380]{G2})}. Every $^\ast$-regular ring is
unit-regular.
\end{conjecture}

This conjecture assumes that the ring is unital. First, we note that it is true for unital Leavitt path
algebras. Let us assume that a unital $L_K(E)$ is $^*$-regular. Then $E$ is acyclic by Theorem
\ref{characterizationtheorem}. Then we have that $L_K(E)$ is unit-regular by \cite[Theorem 2]{AR}.

To prove that the conjecture remains true for Leavitt path algebras of arbitrary graphs, we
adapt the notion of unit-regularity for rings with local units, as was done in \cite{AR} for instance.

Recall that a ring $R$ with identity is said to be \emph{unit-regular} if for each $a\in R$, there is a unit
(an invertible element) $u$ such that $aua=a$. If $R$ is a ring with local units, then $R$ is called
\emph{locally unit-regular} if for each $a\in R$ there is an idempotent (a local unit) $v$ and local inverses
$u,u'$ such that $uu^{\prime }=v=u^{\prime }u$, $va=av=a$ and $aua=a$.

Clearly, a unit-regular (unital) ring is locally unit-regular (take the idempotent $v$ from the definition of
locally unit-regular to be the identity). Conversely, if a ring with identity is locally unit-regular, then it
is unit-regular (see also \cite[Lemma 3 (1)]{AR}). To see this, let $a\in R$. Then there is an idempotent $v$
and local inverses $u,u'$ in $vRv$ such that $uu^{\prime }=v=u^{\prime }u$, $va=av=a$ and $aua=a$. Then
$w=u+(1-v)$ and $w^{\prime }=u^{\prime }+(1-v)$ satisfy $ww^{\prime }=1=w^{\prime }w$ and $a=awa$. Hence $R$
is unit-regular.

\begin{corollary}
Let $E$ be an arbitrary graph and let $K$ be a field with involution. Suppose $L_{K}(E)$ is $^\ast$-regular.
Then
\begin{enumerate}[{\rm (i)}]
\item $L_{K}(E)$ is locally unit-regular.
\item If $L_K(E)$ is a unital ring, then $L_{K}(E)$ is unit-regular.
\end{enumerate}
\end{corollary}
\begin{proof}
(i). If $L_{K}(E)$ is $^\ast$-regular, we have that $E$ is acyclic by Theorem \ref{characterizationtheorem}.
Then $L_{K}(E)$ is locally unit-regular by \cite[Theorem 2]{AR}.
\smallskip

(ii) is a consequence of the fact that every unital locally unit-regular ring is unit-regular.
\end{proof}

Let us also note that the converse of Handelman's Conjecture is not true. The examples of unit-regular and not
$^*$-regular rings can be found in the class of Leavitt path algebras as well. For instance, ${\mathbb
M}_2({\mathbb C})$ with the identity involution on ${\mathbb C}$ is such an example: we know it is not
$^*$-regular but it is unit-regular (as a semisimple ring, see \cite[page 38]{G2}).

\section*{Acknowledgments}

The authors thank Gene Abrams for his valuable discussions during the preparation of this paper. The first
author was partially supported by the Spanish MEC and Fondos FEDER through project MTM2007-60333, and by the
Junta de Andaluc\'{\i}a and Fondos FEDER, jointly, through projects FQM-336 and FQM-2467.



\begin{thebibliography}{10}

\bibitem{Cuntz} \textsc{J. Cuntz}, Simple $C^\ast$-algebras generated by isometries, \emph{Comm. Math.
Phys. } \textbf{57} (1977), 173--185.

\bibitem{Leavitt} \textsc{W. G. Leavitt}, Modules without invariant basis number, \emph{Proc. Amer. Math.
Soc.} \textbf{8} (1957), 322--328.

\bibitem{AA1} \textsc{G. Abrams, G. Aranda Pino}, The Leavitt path algebra of a graph, \emph{J. Algebra}
\textbf{293 (2)} (2005), 319--334.

\bibitem{AMP} \textsc{P. Ara, M.A. Moreno, E. Pardo}, Nonstable $K$-Theory for graph algebras, \emph{Algebr.
Represent. Theory} \textbf{10} (2007), 157--178.

\bibitem{APS} \textsc{G. Aranda Pino, E. Pardo, M. Siles Molina}, Exchange Leavitt path algebras and stable
rank, \emph{J. Algebra} \textbf{305 (2)} (2006), 912--936.

\bibitem{AMMS1} \textsc{G. Aranda Pino, D. Mart\'\i n Barquero, C. Mart\'\i n Gonz\'alez, M. Siles Molina},
The socle of a Leavitt path algebra, \emph{J. Pure Appl. Algebra} \textbf{212 (3)} (2008), 500--509.

\bibitem{AT} \textsc{G. Abrams, M. Tomforde}, Isomorphism and Morita equivalence of graph algebras,
\emph{Trans. Amer. Math. Soc.}  \textbf{363} (2011), 3733 -- 3767.

\bibitem{AbAnhLouP} \textsc{G. Abrams, P. N. \'{A}nh, A. Louly, E. Pardo}, The classification question for
Leavitt algebras, \emph{J. Algebra} \textbf{320} (2008), 1983--2026.

\bibitem{ABC} \textsc{P. Ara, M. Brustenga, G. Corti\~{n}as}, $K$-theory for Leavitt path algebras,
\emph{M\"unster J. of Math} \textbf{2}(2009), 5--33.

\bibitem{Tomforde} \textsc{M. Tomforde}, Uniqueness theorems and ideal structure for Leavitt path algebras,
\emph{J. Algebra} \textbf{318 (1)} (2007), 270--299.

\bibitem{G}  \textsc{K. R. Goodearl}, Leavitt path algebras and direct limits, \emph{Contemp. Math.}
\textbf{480} (2009), 165--187.

\bibitem{Iain} \textsc{I. Raeburn}, Chapter in \emph{Graph algebras: bridging the gap between analysis and
algebra} (G. Aranda Pino, F. Perera, M. Siles Molina, eds.), ISBN: 978-84-9747-177-0, University of
M\'{a}laga Press, M\'alaga, Spain (2007).

\bibitem{AR} \textsc{G. Abrams, K. Rangaswamy}, Regularity conditions for arbitrary Leavitt path algebras,
\emph{Algebr. Represent. Theory}  \textbf{13 (3)} (2010), 319--334.

\bibitem{S} \textsc{M. Siles Molina}, Algebras of quotients of Leavitt path algebras. \emph{J. Algebra}
\textbf{319 (12)} (2008), 5265--5278.

\bibitem{AA2} \textsc{G. Abrams, G. Aranda Pino}, Purely infinite simple Leavitt path algebras, \emph{J.
Pure Appl. Algebra} \textbf{207 (3)} (2006), 553--563.

\bibitem{DT} \textsc{D. Drinen, M. Tomforde}, The $C^*$-algebras of arbitrary graphs, \emph{Rocky Mountain
J. Math} \textbf{35 (1)} (2005), 105--135.

\bibitem{AA3} \textsc{G. Abrams, G. Aranda Pino}, The Leavitt path algebras of arbitrary graphs,
 \emph{Houston J. Math.} \textbf{34 (2)} (2008), 423--442.

\bibitem{Lam} \textsc{T.Y. Lam}, \emph{A First Course on Noncommutative Rings}, Springer-Verlag New York
(1991).

\bibitem{Be} \textsc{S. K. Berberian}, {\em Baer $^*$-rings}, Die Grundlehren der mathematischen
Wissenschaften 195, Springer-Verlag, Berlin-Heidelberg-New York (1972).

\bibitem{AAS1} \textsc{G. Abrams, G. Aranda Pino, M. Siles Molina}, Finite-dimensional Leavitt path
algebras,
 \emph{J. Pure Appl. Algebra.} \textbf{209 (3)} (2007), 753--762.

\bibitem{G2} \textsc{K. R. Goodearl}, \emph{Von Neumann Regular Rings}, Second Ed., Krieger, Malabar, FL
(1991).















\end{thebibliography}
\end{document}